\documentclass[11pt]{scrartcl}
\RequirePackage{enumerate}
\RequirePackage{verbatim}
\RequirePackage{hyperref}

\RequirePackage{amsmath,amsfonts,amssymb}
\RequirePackage[mathcal,mathbf]{euler}

\usepackage{latexsym,amscd,amsthm}
\usepackage[all]{xy}

\usepackage[dvips]{epsfig}
\usepackage[dvips]{graphics}

\newtheorem{thm}{Theorem}[section]
\newtheorem{lem}[thm]{Lemma}

\newtheorem{prop}[thm]{Proposition}
\newtheorem{prop-def}[thm]{Proposition-Definition}

\theoremstyle{definition}

\newtheorem{rem}[thm]{Remark}

\newtheorem{rems}[thm]{Remarks}

\newcommand\oUA{\!\!\underset{U(A)}{\otimes}\!\!}
\newcommand\oUL{\!\!\underset{U(L)}{\otimes}\!\!}
\newcommand\oR{\!\underset{R}{\otimes}\!}

\author{Damien Calaque}
\title{A PBW theorem for inclusions of (sheaves of) \\ Lie algebroids}
\date{}
\begin{document}

\maketitle

\begin{abstract}
\noindent{\bf Abstract.} Inspired by the recent work of Chen-Sti\'enon-Xu on {\it Atiyah classes} associated 
to inclusions of Lie algebroids, we give a very simple criterium (in terms of those classes) for relative 
Poincar\'e-Birkhoff-Witt type results to hold. The tools we use (e.g.~the {\it first infinitesimal 
neighbourhood Lie algebroid}) are straightforward generalizations of the ones previously developped by 
C\u ald\u araru, Tu and the author for Lie algebra inclusions. 
\end{abstract}

{\footnotesize \tableofcontents}

\section{Introduction}

\subsection{General context}

This paper is part of a more general project which aims at building a dictionnary between Lie theory and algebraic geometry. 

In \cite{AC} Arinkin and C\u{a}ld\u{a}raru provide a necessary and sufficient condition for the Ext-algebra of a closed subvariety 
$X$ (a ``brane'') of an algebraic variety $Y$ to be isomorphic, as an object of the derived category of $X$, to $S(N[-1])$, where 
$N$ is the normal bundle of $X$ into $Y$ ; the condition is that $N$ can be lifted to the first infinitesimal neighbourhood $X^{(1)}$. 
This condition is equivalent to the vanishing of a certain class in $Ext^2_{\mathcal O_X}(N^{\otimes 2},N)$. 

This result has been translated into Lie theory in \cite{CCT} by C\u{a}ld\u{a}raru, Tu and the author as follows. 
For an inclusion of Lie algebras $\mathfrak h\subset\mathfrak g$, we gave a necessary and sufficient condition 
for $U(\mathfrak g)/U(\mathfrak g)\mathfrak h$ to be isomorphic, as an $\mathfrak h$-module, to $S(\mathfrak g/\mathfrak h)$ ; 
the condition is that the quotient module $\mathfrak n=\mathfrak g/\mathfrak h$ extends to a Lie algebra $\mathfrak h^{(1)}$ 
``sitting in between'' $\mathfrak h$ and $\mathfrak g$. 
Similarly, this condition is equivalent to the vanishing of a certain class in $Ext^1_{\mathfrak h}(\mathfrak n^{\otimes 2},\mathfrak n)$. 

It is Kapranov who observed in \cite{Ka} that the shifted tangent bundle $T_X[-1]$ of an algebraic variety $X$ 
is a Lie algebra object in the derived category of $X$, with Lie bracket being given by the Atiyah class of $T_X[-1]$. 
Moreover, any object of the derived category becomes a representation of this Lie algebra {\it via} its own Atiyah class. 
In the case of a closed embedding $i:X\hookrightarrow Y$ we then get an inclusion of Lie algebras objects $T_X[-1]\subset i^*T_Y[-1]$, 
so that the main result of \cite{AC} can be deduced, in principle, from a version of the main result of \cite{CCT} that would hold 
in a triangulated category. We refer to the introduction of \cite{CCT} and to \cite{Cal} for more details on this striking analogy. 

In \cite{CCT2} we have exhibited a Lie algebroid structure on the shifted normal bundle $N[-1]$ of a closed embedding $X\hookrightarrow Y$. 
Extending the results of \cite{CCT} to the Lie algebroid setting then seems quite natural, especially if one wants to understand the geometry of a sequence of closed embeddings $X\hookrightarrow Y\hookrightarrow Z$. 

\subsection{Description of the main results}

For simplicity of exposition we assume in this introduction that ${\bf k}$ is a field of characteristic zero. 
We let $X$ be a topological space equipped with a sheaf of ${\bf k}$-algebras $\mathcal R$, and 
$\mathcal A\subset\mathcal L$ be an inclusion of sheaves of Lie algebroids over $\mathcal R$ (we refer to Section 
\ref{sec-2} for standard Definitions), which are locally free as $\mathcal R$-modules. The locally free 
$\mathcal R$-module $\mathcal L/\mathcal A$ turns out to be naturally equipped with an action of $\mathcal A$ 
(see \S~\ref{sec-L/A}), also-known-as a flat $\mathcal A$-connection. 

In \cite{CSX} Chen, Sti\'enon and Xu introduce a very interesting class 
$\alpha_{\mathcal E}\in{\rm Ext}^1_{\mathcal A}\big((\mathcal L/\mathcal A)\otimes_{\mathcal R}\mathcal E,\mathcal E\big)$, 
for any $\mathcal A$-module $\mathcal E$, which is the obstruction to the existence of a lift of the flat 
$\mathcal A$-connection on $\mathcal E$ to a possibly non-flat $\mathcal L$-connection. They define this class 
in geometric terms, while we provide in this paper a purely algebraic description of $\alpha_{\mathcal E}$ 
(see \S~\ref{sec-alpha}) which makes sense in a wider context. 

Inspired by a previous work \cite{CCT} of the author with C\u ald\u araru and Tu we also introduce a new Lie algebroid 
$A^{(1)}$, called the {\it first infinitesimal neighbourhood Lie algebroid}, which fits in between $\mathcal A$ and 
$\mathcal L$ in the sense that we have a sequence of Lie algebroid morphisms 
$\mathcal A\longrightarrow\mathcal A^{(1)}\longrightarrow \mathcal L$. 

We then prove the following result, which generalizes \cite[Theorem 1.3]{CCT}: 
\begin{thm}\label{thm-main}
The following statements are equivalent: 

(1) The class $\alpha_{\mathcal L/\mathcal A}$ vanishes. 

(2) The $\mathcal A$-module structure on $\mathcal L/\mathcal A$ lifts to an $A^{(1)}$-module structure. 

(3) $U(\mathcal L)/U(\mathcal L)\mathcal A$ is isomorphic, as a filtered $\mathcal A$-module, 
to $S_{\mathcal R}(\mathcal L/\mathcal A)$. 
\end{thm}
Here $S$ and $U$ denote the symmetric and the universal enveloping algebra, respectively. 

We can also prove a more general version of the above result for $\mathcal A$-modules other than 
${\bf 1}_{\mathcal{A}}$ (compare with \cite[Theorem 5.1]{CCT}): 
\begin{thm}\label{thm-E0}
Let $\mathcal E$ be an $\mathcal A$-module which is locally free as an $\mathcal R$-module. 
Then the  following statements are equivalent: 

(1) The classes $\alpha_{\mathcal L/\mathcal A}$ and $\alpha_{\mathcal E}$ vanish. 

(2) The $\mathcal A$-module structures on $\mathcal L/\mathcal A$ and $\mathcal E$ lift to $A^{(1)}$-module structures. 

(3) $U(\mathcal L)\otimes_{U(\mathcal A)}\mathcal E$ is isomorphic, as a filtered $\mathcal A$-module, 
to $S_{\mathcal R}(\mathcal L/\mathcal A)\otimes_{\mathcal R}\mathcal E$. 
\end{thm}

\subsection{Plan of the paper}

We start with some recollection about Lie algebroids in Section \ref{sec-2}: e.g.~we recall definitions and basic 
properties of the universal enveloping algebra, the de Rham (or Chevalley-Eilenberg) complex and jets of a given Lie 
algebroid, as well as the construction of free Lie algebroids. 
In Section \ref{sec-3} we introduce natural objects associated to an inclusion of Lie 
algebroids: the associated quotient module, the Chen-Sti\'enon-Xu class \cite{CSX}, and the first infinitesimal 
neighbourhood Lie algebroid. We also interprete the Chen-Sti\'enon-Xu class as the obstruction to extend modules 
to the first infinitesimal neighbourhood. Section \ref{sec-4} and \ref{sec-5} are the heart of the paper: the 
fourth Section is devoted to the statement and the proof of a PBW type theorem for inclusions into first infinitesimal 
neighbourhoods, from which we deduce the general case in the fifth Section. We then extend the previous results to 
general modules in Section \ref{sec-6} (this Section is to Sections \ref{sec-4} and \ref{sec-5} what Theorem 
\ref{thm-E0} is to Theorem \ref{thm-main}). We finally sheafify everything and prove the two Theorems of the 
Introduction in the last Section. We end the paper with an appendix where we prove that two classes coincide. 

\subsection{Notation and conventions}

Unless otherwise specified ${\bf k}$ is a commutative ring, all algebraic structures we 
consider are defined over ${\bf k}$, and all filtrations are ascending, indexed 
by non-negative integers. By an {\it $n$-filtered morphism} we mean a morphism that raises the 
filtration degree by $n$. A {\it filtered morphism} (a-k-a morphism of filtered objects) is a 
$0$-filtered morphism. 

We now describe our conventions regarding tensor products. For a commutative ring $R$, we write 
$\otimes_R$ for the tensor product of {\bf left} $R$-modules. As there is no ambiguity we define 
$\otimes:=\otimes_{\bf k}$. For a (possibly noncommutative) ring $B$, we denote by $\underset{B}{\otimes}$ 
the tensor product between right and left $B$-modules. 

For a left $R$-module $M$ we denote by $S_R(M)$, resp.~$T_R(M)$, the symmetric, resp.~tensor, algebra generated 
by $M$ over $R$. Both are considered as graded $R$-algebras; however, we don't require $R$ to be central in 
$R$-algebras. We write $S_R^k(M)$, resp.~$T_R^k(M)$, for the $k$-the homogeneous component. 

\subsubsection*{Acknowledgements}

I thank Darij Grinberg for his careful reading of a preliminary version of this paper.  
His comments helped me correct some mistakes and improve the exposition.  
I also thank the anonymous referee for her/his extremely valuable suggestions. 
This project has been partially supported by a grant from the Swiss National Science Foundation (project number $200021\underline{~}137778$). 

\section{Lie algebroids and associated structures}\label{sec-2}

Let $R$ be a commutative ${\bf k}$-algebra and $L$ a {\it Lie algebroid} over $R$, 
which means that the pair $(R,L)$ is a Lie-Rinehart algebra (see \cite{R}). Namely, $L$ is a Lie 
${\bf k}$-algebra equipped with an $R$-module structure and an $R$-linear Lie algebra map 
$\rho:L\to {\rm Der}_{\bf k}(R)$ such that $[l,rl']=r[l,l']+\rho(l)(r)l'$ for $l,l'\in L$ 
and $r\in R$. The map $\rho$ is called the {\it anchor map} and we usually omit its symbol 
from the notation: for $l\in L$ and $r\in R$, we write $l(r):=\rho(l)(r)$. In particular 
$R\oplus L$ inherits the structure of a Lie ${\bf k}$-algebra with bracket given by 
$[(r,l),(r',l')]=(l(r')-l'(r),[l,l'])$, for $r,r'\in R$ and $l,l'\in L$.  

What we discuss in this Section is relatively standard and can be found e.g.~in \cite{R,X,Ka2} 
and references therein (perhaps phrased in a different way). 

\subsection{The universal enveloping algebra of a Lie algebroid}\label{sec-2.1}

We define the enveloping algebra $U(R,L)$ of the pair $(R,L)$ to be the quotient 
of positive part of the universal enveloping algebra\footnote{By this we mean the subalgebra generated by 
$R\oplus L$ (i.e.~the kernel of the natural ${\bf k}$-augmentation). } of the Lie ${\bf k}$-algebra 
$R\oplus L$ by the following relations: $r\otimes l=rl$ ($r\in R$, $l\in R\oplus L$). As there 
is no risk of confusion we simply write $U(L)$ for $U(R,L)$, which is obviously an $R$-algebra {\it via} 
the natural map $R\longrightarrow U(L)$. It therefore inherits an $R$-bimodule structure. 

It turns out that $U(L)$ is also a cocommutative coring in {\bf left} $R$-modules\footnote{We would like 
to warn the reader that the multiplication is defined on $U(L)\oR U(L)$ while the comultiplication takes 
values in $U(L)\otimes_R U(L)$, where only the left $R$-module structure is used. }. 
Namely, the coproduct $\Delta:U(L)\longrightarrow U(L)\otimes_R U(L)$ is the multiplicative map defined 
on generators by $\Delta(r)=r\otimes 1=1\otimes r$ ($r\in R$) and $\Delta(l)=l\otimes1+1\otimes l$ ($l\in L$). 
The anchor map can be extended to an $R$-algebra morphism $U(L)\longrightarrow {\rm End}(R)$ (actually taking values 
in the ring ${\rm Diff}(R)$ of differential operators) sending $r\in R$ to the multiplication by $r$ and $l\in L$ to $\rho(l)$. The counit 
$\epsilon:U(L)\longrightarrow R$ is defined by $\epsilon(P):=P(1)$. 
\begin{rem}
The above definition of $\Delta$ needs some explanation. Being the quotient of $U(L)\otimes U(L)$ 
by the right ideal generated by $r\otimes 1-1\otimes r$ $(r\in R$), $U(L)\otimes_R U(L)$ is {\bf not} an algebra. 
Nevertheless, one easily sees that $r\otimes 1$ ($r\in R$) and $l\otimes1+1\otimes l$ ($l\in L$) sit in the normalizer 
of that ideal, so that multiplying them together makes perfect sense. 
\end{rem}
In what follows, left $U(L)$-modules are called $L$-modules. We say that a given (left) $R$-module $E$ is acted 
on by $L$ if it is equipped with an $L$-module structure of which the restriction to $R$ gives back the original 
$R$-action we started with. The abelian category $L$-$mod$ of $L$-modules is monoidal, with product being $\otimes_R$ 
(and $U(L)$ acting on a tensor product {\it via} the coproduct) and unit ${\bf 1}_L$ being $R$ equipped with 
the action given by the anchor $\rho$. 

\medskip

Any morphism $f:L\longrightarrow L'$ of Lie algebroids over $R$ automatically induces a morphism 
of algebras $U(L)\longrightarrow U(L')$ which preserves all the above algebraic structures. We denote the 
restriction (or pull-back) functor $L'\textrm{-}mod\longrightarrow L\textrm{-}mod$ by $f^*$, and by 
$f_!:=U(L')\oUL -$ its left adjoint. Notice that $f^*$ is monoidal, while $f_!$ is not ($f_!$ is only colax-monoidal). 

\medskip

There is a canonical filtration on $U(L)$ obtained by assigning degree $0$, resp.~$1$, to elements of 
$R$, resp.~$L$. All structures we have defined so far on $U(L)$ respect this filtration. If, additionally, $L$ is 
itself equipped with a filtration, then this filtration extends to $U(L)$. 
The canonical filtration on $U(L)$ can be seen as coming from the obvious ``constant'' filtration on $L$ (the only degree $0$ 
element is $0$ and all elements in $L$ are of degree $\leq1$). 

\begin{rem}\label{rem-2.2}
One can alternatively describe the functor $U$ as a left adjoint. 
Namely, we consider the category of {\it anchored algebras}: they are defined as $R$-algebras $B$ equipped 
with an $R$-algebra morphism $\rho:B\longrightarrow {\rm End}(R)$, where the $R$-algebra structure on ${\rm End}(R)$ 
is the given by $r\longmapsto(l_r:b\mapsto rb)$. 
There is a functor $Prim$ from anchored algebras to Lie algebroids that sends an anchored algebra $B$ to the sub-$R$-module 
consisting of those elements $b\in B$ such that $\rho(b)\in {\rm Der}(R)$. 
We then have an adjuntion 
$$U:\{\textrm{Lie algebroids}\}\begin{matrix}\,\longrightarrow \\[-0.3cm] \longleftarrow\,\end{matrix}\{\textrm{anchored algebras}\}:Prim\,.$$
\end{rem}

\subsection{The de Rham complex of a Lie algebroid}

To any $L$-module $E$ we associate the complex of graded $R$-modules $C^\bullet(L,E)$, consisting of 
${\rm Hom}_R(\wedge_R^\bullet L,E)$ equipped with the differential $d$ defined as follows: for 
$\omega\in C^n(L,E)$ and $l_0,\dots,l_n\in L$, 
$$
d(\omega)(l_0,\dots,l_n):=\sum_{i=0}^n(-1)^il_i\omega\big(l_0,\dots,\widehat{l_i},\dots,l_n\big)
+\sum_{i<j}(-1)^{i+j}\omega\big([l_i,l_j],l_0,\dots,\widehat{l_i},\dots,\widehat{l_j},\dots,l_n\big)\,.
$$
The map that associates to $l\in L$ the element $\nabla_l\in{\rm End}_{{\bf k}}(E)$ defined by 
$\nabla_l(e):=d(e)(l)$ is sometimes called a flat connection. It completely determines both the differential 
$d$ and the $L$-action on $E$. 

\medskip

We have the following functoriality property: for $f:L\to L'$ a morphism of Lie algebroids over $R$ and 
$\varphi:E\to F$ a morphism of $L'$-modules, we have an 
obvious $R$-linear map $f^*\varphi:C^\bullet(L',E)\longrightarrow C^\bullet(L,f^*F)$ defined by 
$(f^*\varphi)(\omega):=\varphi\circ\omega\circ f$. 
We also have that for any two $L$-modules $E$ and $F$, there is a product 
$C^\bullet(L,E)\otimes C^\bullet(L,F)\longrightarrow C^\bullet(L,E\otimes_R F)$. 
In particular, this turns $C^\bullet(L):=C^\bullet(L,{\bf 1}_L)$ into a differential graded 
commutative $R$-algebra. 

\subsection{Lie algebroid jets}

For any $L$-module $E$ we define the $L$-module $J_L(E)$ of {\it $L$-jets}, or simply {\it jets}, as the internal 
Hom ${\rm Hom}_R\big(U(L),E\big)$ from the universal enveloping algebra $U(L)$ to $E$. 

This requires some explanation. First of all observe that the monoidal category $L$-$mod$ is closed. 
The internal Hom of two $L$-modules $E$ and $F$ is given by the $R$-module ${\rm Hom}_R(E,F)$ equipped with 
the following $L$-action: for $l\in L$, $\psi:E\to F$ and $e\in E$, 
$(l\cdot\psi)(e):=l\cdot\big(\psi(e)\big)-\psi(l\cdot e)$. 
Then $U(L)$ is naturally an $L$-module (being a left $U(L)$-module over itself). 

\medskip

But $U(L)$ is actually an $U(L)$-bimodule. Therefore, $J_L(E)$ inherits a second left $U(L)$-module structure, 
denoted $*$, which commutes with the above one and is defined in the following way: 
for $\phi\in J_L(E)$ and $P,Q\in U(L)$, $(P*\phi)(Q)=\phi(QP)$. 
When $E={\bf 1}_L$ the two commuting $L$-module structures one gets on $J_L:=J_L({\bf 1}_E)$ are precisely 
the ones described in \cite[\S~4.2.5]{CVdB}. 

\begin{rem}\label{rem-bimod}
This is actually true for any $U(L)$-$U(L')$-bimodule $M$: the space ${\rm Hom}_R(M,E)$ has an $L$-module and 
an $L'$-module structures that commute\footnote{Notice that even the two underlying $R$-module structures are 
different. }. In particular, the space ${\rm Hom}_{L\textrm{-}mod}(M,E)$ itself is naturally an $L'$-module. 
\end{rem}

\subsection{Free Lie algebroids (after M.~Kapranov)}\label{ss-free}

Let us first recall from \cite{Ka2} that an $R$-module $M$ is {\it anchored} if it is equipped with an $R$-linear map 
$\rho:M\longrightarrow {\rm Der}(R)$, called the {\it anchor map}. 
Like for Lie algebroids we usually omit the symbol $\rho$ from our notation: for $m\in M$ and $r\in R$, 
we write $m(r):=\rho(m)(r)$. There is an obvious forgetful functor which goes from the category of 
Lie algebroids to that of anchored modules, that forgets everything except the underlying $R$-module 
structure of the Lie algebroid and the anchor map. This functor has a left adjoint, denoted $FR$. 

For any anchored $R$-module $M$ we call $FR(M)$ the {\it free Lie algebroid} generated by $M$. 
It can be described in the following way, as a filtered quotient of the free Lie ${\bf k}$-algebra 
$FL(M)$ generated by $M$. First of all, by adjunction $\rho$ naturally extends to a Lie ${\bf k}$-algebra morphism 
$FL(M)\longrightarrow {\rm Der}(R)$. 
Then we define $FR(M)$ as the quotient of $FL(M)$ by the following relations: for $r\in R$, $m\in FL(M)$, 
and $m'\in FL(M)$, $[m,rm']-[rm,m']=m(r)m'+m'(r)m$. These relations being satisfied in ${\rm Der}(R)$ then 
$\rho$ factors through $FR(M)$. 
Finally, we define an $R$-module structure on $FR(M)$ in the following way: $r[m,m']:=[m,rm']-m(r)m'=[rm,m']+m'(r)m$. 

\medskip 

According to Remark \ref{rem-2.2} we then have a sequence of adjunctions
$$
\{\textrm{anchored modules}\}\,
\begin{matrix}\overset{FR}{\longrightarrow} \\[-0.3cm] \longleftarrow\end{matrix}\,
\{\textrm{Lie algebroids}\}\,
\begin{matrix}\overset{U}{\longrightarrow} \\[-0.3cm] \underset{Prim}{\longleftarrow}\end{matrix}\,
\{\textrm{anchored algebras}\}\,.
$$

\begin{rem}
The above sequence of adjunctions extends to filtered versions. 
Unless otherwise specified, the canonical filtration we put on an anchored module $M$ is the ``constant'' one we already mentioned 
in \S~\ref{sec-2.1}. Then the associated graded of the induced filtration on $FR(M)$ is the free Lie $R$-algebra $FL_R(M)$ generated by $M$, 
and the associated graded of the induced filtration on $U\big(FR(M)\big)$ is $U\big(FL_R(M)\big)=T_R(M)$. 
\end{rem}

\section{Structures associated to inclusions of Lie algebroids}\label{sec-3}

Let $R$ be a commutative ${\bf k}$-algebra and $i:A\hookrightarrow L$ an inclusion of Lie algebroids over $R$. 

\subsection{The $A$-module $L/A$}\label{sec-L/A}

It is well-known that $A$ does not necessarily act on itself (meaning that $A$ is not an $A$-module in any natural way). 
In this paragraph we consider the quotient $R$-module $L/A$ and define an $A$-action on it in the following way: 
for any $a\in A$ and any $l\in L$, we define $a\cdot(l+A):=[a,l]+A$ (when there is no ambiguity we omit 
the inclusion symbol $i$ from the notation). The only nontrivial identity to check is that $(ra)\cdot=r(a\cdot)$: 
$$
(ra)\cdot(l+A)=[ra,l]+A=r[a,l]-\underbrace{l(r)a}_{\in A}+A=r[a,l]+A=r\big(a\cdot(l+A)\big)\,.
$$

\medskip

From now and in the rest of the paper we make the following assumption: \\

\indent\indent\indent {\it The map $L\longrightarrow U(L)$ is injective. }
\hfill $(\O)$

\subsection{The extension class $\alpha$ (inspired by Chen-Sti\'enon-Xu)}\label{sec-alpha}

Let $E$ be an $A$-module. We define a class $\alpha_E\in{\rm Ext}^1_{A}\big((L/A)\otimes_RE,E\big)$, 
which generalizes the one introduced in \cite{CCT} for Lie algebras, {\it via} the following short exact 
sequence of $A$-modules: 
\begin{equation}\label{eq-alphabis}
0\longrightarrow E\longrightarrow \left(U(L)\oUA E\right)^{\leq1}\longrightarrow L/A\otimes_RE\longrightarrow 0\,.
\end{equation}
We have to explain why the middle term in \eqref{eq-alphabis} is an $A$-module, which is {\it a priori} not guaranteed. 
Namely, even though $U(L)$ is an $A$-module ({\it via} left multiplication) its filtered pieces $U(L)^{\leq k}$ are 
not (because $AU(L)^{\leq k}\subset U(L)^{\leq k+1}$). Nevertheless, $U(L)\oUA E$ turns out to be 
a filtered $A$-module because of the following: for $a\in A$, $P\in U(L)^{\leq k}$ and $e\in E$, 
\begin{equation}\label{eq-AmodU}
a(P\otimes e)=aP\otimes e=\big([a,P]+Pa\big)\otimes e=[a,P]\otimes e+P\otimes ae\in (U(L)\oUA E)^{\leq k}\,.
\end{equation}

We set $\alpha:=\alpha_{L/A}$. 

\subsubsection*{Relation to Atiyah classes as they are defined by Chen-Sti\'enon-Xu in \cite{CSX}}

We consider the filtered subspace $J_{L/A}(E)$ of $J_L(E)$ consisting of those maps 
$\phi:U(L)\to E$ which are $A$-linear: for $Q\in U(A)$ and $P\in U(L)$, 
\begin{equation}\label{eq-comp2}
\phi(QP)=Q\cdot\phi(P)\,.
\end{equation}
According to Remark \ref{rem-bimod} there is a residual $L$-module structure on $J_{L/A}(E)$: 
for $Q\in U(L)$, we have $(Q*\phi)(P)=\phi(PQ)$. 
Even though the successive quotients $J_L^n(E):={\rm Hom}_R\big(U(L)^{\leq n},E\big)$ of $J_L(E)$ are not 
$U(L)$-bimodules, it turns out that their subspaces $J_{L/A}^n(E)$ inherits an $A$-action from the above 
residual $L$-action. Namely, for $Q\in U(A)$, $\phi\in J_{L/A}(E)$ and $P\in U(L)$ we have
\begin{equation}\label{eq-Amod}
(Q*\phi)(P)=\phi(PQ)=\phi([P,Q]+QP)=\phi([P,Q])+Q\cdot\big(\phi(P)\big)\,.
\end{equation}
Therefore $Q*$ descends to $J_{L/A}^n(E)$. We then have the following exact sequence of $A$-modules: 
\begin{equation}\label{eq-alpha}
0\longrightarrow {\rm Hom}_R(L/A,E)\longrightarrow J^1_{L/A}(E)\longrightarrow E\longrightarrow 0\,.
\end{equation}
This determines a class $\widetilde\alpha_E\in{\rm Ext}^1_A\big(E,{\rm Hom}_R(A,E)\big)$, which has been 
first defined in a differential geometric context by Chen-Sti\'enon-Xu in \cite[\S~2.5.1]{CSX}. 
\begin{prop}\label{prop-sameclasses}
The images of the classes $\alpha_E$ and $\widetilde\alpha_E$ coincide in 
$$
{\rm Hom}_{D(A)}\big((L/A)\overset{\mathbb{L}}{\otimes}_RE,E[1]\big)
\cong{\rm Hom}_{D(A)}\big(E,\mathbb{R}{\rm Hom}_R(L/A,E)[1]\big)\,,
$$
where $D(A)$ denotes the bounded derived category of $A$-modules. 
\end{prop}
\begin{proof}
See Appendix. 
\end{proof}
In particular this implies that $\alpha_E=0$ if and only if $\widetilde\alpha_E=0$ (which we also prove in Proposition \ref{prop-3.2}). 
This is because the natural maps from ${\rm Ext}^1_A\big((L/A)\otimes_RE,E)$ and ${\rm Ext}^1_A\big(E,{\rm Hom}_R(L/A,E)\big)$  to 
${\rm Hom}_{D(A)}\big((L/A)\overset{\mathbb{L}}{\otimes}_RE,E[1]\big)$ and 
${\rm Hom}_{D(A)}\big(E,\mathbb{R}{\rm Hom}_R(L/A,E)[1]\big)$, respectively, are both injective. 
\begin{rem}\label{rem-Koszul}
It is very likely that $\tilde\alpha_E$ fits into the framework of Atiyah classes associated 
to dDg algebras \cite[Section 8]{CVdB}, but proving it would require to understand deeply the 
Koszul-type duality between $C^\bullet(L)$ and $U(L)$ (see e.g.~\cite{P} in which the two extreme 
cases $L={\rm Der}(R)$ and $R={\bf k}$ are covered). 
\end{rem}

\subsection{The first infinitesimal neighbourhood Lie algebroid $A^{(1)}$}

Being a Lie algebroid over $R$, $L$ is in particular an anchored $R$-module. 
We can therefore consider the free Lie algebroid $FR(L)$ over $R$ generated by $L$. 
Let us then consider the quotient $A^{(1)}$ of $FR(L)$ by the ideal generated 
by\footnote{Observe that, contrary to what is suggested by the notation, $A^{(1)}$ does not only depend on $A$ but also on $L$. } 
$$
[a,l]_{FR(L)}-[a,l]_{L},\quad a\in A,l\in L\,.
$$
Observe that it is a well-defined (Lie algebroid) ideal in $FR(L)$ as the anchor map of $FR(L)$ coincides by definition 
with the one of $L$ on generators. We call $A^{(1)}$ the {\it first infinitesimal neighbourhood} of $A$ (see \cite{CCT}, 
where the geometric motivation behind such a denomination is given). We denote by $j$ the Lie algebroid inclusion of 
$A$ into $A^{(1)}$. We now prove that, for an $A$-module $E$, the classes $\alpha_E$ and $\widetilde\alpha_E$ both 
give the obstruction to lift $E$ to an $A^{(1)}$-module: 
\begin{prop}\label{prop-3.2}
Let $E$ be an $A$-module. Then the following statements are equivalent: 

(1) There exists an $A^{(1)}$-module $E^{(1)}$ such that $j^*\big(E^{(1)}\big)=E$. 

(2) $\alpha_E=0$. 

(3) $\widetilde\alpha_E=0$.
\end{prop}
From now we omit the symbol $\cdot$ in formul\ae. We also omit parentheses when they are not strictly necessary 
(e.g.~when one can use associativity). 
\begin{proof}
(1)$\implies$(2): assume that such an $A^{(1)}$-module $E^{(1)}$ exists. 
Observe that any $P\in U^{\leq 1}(L)=R\oplus L$ lies in $U(FR(L))$, and thus acts on $E$ through the 
quotient map $U(FR(L))\to U(A^{(1)})$. We therefore set, for any $e\in E$, $s(P\otimes e):=P e$. 
Since $E=j^*\big(E^{(1)}\big)$ then there is no ambiguity in the way $Q\in U(A)$ acts on $E$, so that 
$s(PQ\otimes e)=PQe=s(P\otimes Qe)$. 
Moreover, for the same reason $s$ is $A$-linear:  
$$
Qs(P\otimes e)=QPe=[Q,P]_{A^{(1)}}e+PQe=[Q,P]_Le+PQe=s\big(Q(P\otimes e)\big)\,.
$$
Therefore $s$ is a splitting \eqref{eq-alphabis}. 

(2)$\implies$(3): assume $\alpha_E=0$. Then there exists a splitting 
$s:\left(U(L)\underset{U(A)}\otimes E\right)^{\leq1}\longrightarrow E$ of \eqref{eq-alphabis}. 
We define a splitting $\widetilde{s}:E\longrightarrow J^1_{L/A}E$ of \eqref{eq-alpha} as follows: 
for $P\in U(L)^{\leq1}$ and $e\in E$, $\widetilde{s}(e)(P):=s(P\otimes e)$. We must check two things: 
\begin{itemize}
\item $\widetilde{s}(e)$ belongs to $J_{L/A}(E)$, i.e.~satisfies \eqref{eq-comp2}: if $Q\in U(A)$ then 
$$
\widetilde{s}(e)(QP)=s(QP\otimes e)=Q\big(s(P\otimes e)\big)=Q\big(\widetilde s(e)(P)\big)\,.
$$
\item $\widetilde{s}$ is $A$-linear: if $Q\in U(A)$ then 
$\big(Q*\widetilde{s}(e)\big)(P)=\widetilde{s}(e)(PQ)=s(PQ\otimes e)=s(P\otimes Qe)=\widetilde{s}(Qe)(P)$. 
\end{itemize}

(3)$\implies$(1): finally assume we have a section $s:E\to J_{L/A}(E)$ of \eqref{eq-alpha}. Then for any $l\in L$ and any $e\in E$ 
we define $le:=s(e)(l)$. We first observe that this defines an action of $FR(L)$ on $E$: 
\begin{itemize}
\item on the one hand if $r\in R$, $l\in L$, $e\in E$ then $(rl)e=s(e)(rl)=r\big(s(e)(l)\big)=r(le)$ . 
\item on the other hand if $r\in R$, $l\in L$, $e\in E$ then 
$$
l(re)-r(le)=s(re)(l)-s(e)(rl)=\big(r*s(e)\big)(l)-s(e)(rl)=s(e)(lr-rl)=s(e)\big(l(r)\big)=l(r)e\,.
$$
\end{itemize}
This action restricts to the one of $A\subset FR(L)$: for $a\in A$ and $e\in E$, $s(e)(a)=as(e)(1)=ae$. 
Finally we see that it descends to an action of $A^{(1)}$: for $a\in A$, $l\in L$ and $e\in E$, 
$$
\big([l,a]_L\big)e=s(e)\big([l,a]_L\big)=\big(a*s(e)\big)(l)-a\big(s(e)(l)\big)
=\big(s(ae)\big)(l)-ale=lae-ale=\big([l,a]_{FR(L)}\big)e\,.
$$
In the second equality we have used formula \eqref{eq-Amod} for the $A$-action on $J_{L/A}(E)$ and in the third 
one we have used $A$-linearity of $s$. 
\end{proof}

\section{PBW for the inclusion into the first infinitesimal neighbourhood}\label{sec-4}

In this Section we prove a version of the main Theorem for the inclusion $j:A\hookrightarrow A^{(1)}$. 
The proof we give follows very much and hopefully simplifies the one of Darij Grinberg for Lie algebras (see \cite{G}). 
It is very likely that a proof using some Koszulness property in the spirit of \cite{CCT} might also 
exist\footnote{But it would have required to adapt some standard but quite technical constructions to the context 
of non-central $R$-algebras (see Remark \ref{rem-Koszul}). }. 

\medskip

The goal of the present Section is to understand the $A$-module 
$$
j^*j_!({\bf 1}_A):=U\big(A^{(1)}\big)\oUA {\bf 1}_A=U\big(A^{(1)}\big)/U\big(A^{(1)}\big)A\,.
$$
According to \S~\ref{ss-free} the free Lie algebroid $FR(L)$ and its quotient $A^{(1)}$ admit a filtration. 
Their universal enveloping algebras are therefore filtered too in an obvious way. We denote these filtrations 
by $F^kU\big(FR(L)\big)$ and $F^kU\big(A^{(1)}\big)$ in order to distinguish them from the standard filtrations on 
universal enveloping algebras. 
It is worth noticing that the $A$-module structure on $j^*j_!({\bf 1}_A)$ is compatible with the induced filtration 
$F^k:=F^k\big(j^*j_!({\bf 1}_A)\big)$: for any $a\in A$ and any $P\in F^kU\big(A^{(1)}\big)$, we have
$$
a\Big(P+U\big(A^{(1)}\big)A\Big)=aP+U\big(A^{(1)}\big)A=[a,P]+U\big(A^{(1)}\big)A\subset F^kU\big(A^{(1)}\big)+U\big(A^{(1)}\big)A\,.
$$

According to \S~\ref{ss-free} the associated graded algebra of the filtered $R$-algebra $U\big(FR(L)\big)$ is the tensor 
$R$-algebra $T_R(L)$. The filtered $R$-linear surjection $\xi:U\big(FR(L)\big)\to j^*j_!({\bf 1}_A)$ therefore 
induces a graded $R$-linear map ${\rm gr}(\xi):T_R(L)\to {\rm gr}\big(j^*j_!({\bf 1}_A)\big)$. We shall also use 
the graded $R$-algebra surjection $\pi:T_R(L)\to T_R(L/A)$. 
\begin{thm}\label{thm-important}
The class $\alpha=\alpha_{L/A}$ vanishes if and only if there exists an isomorphism of filtered $A$-modules 
$\varphi:j^*j_!({\bf 1}_A)\longrightarrow T_R(L/A)$ such that ${\rm gr}(\varphi\circ\xi)=\pi$. Moreover, when this 
happens one can choose $\varphi$ so that it is $A^{(1)}$-linear. 
\end{thm}
We devote the rest of this Section to the proof of Theorem \ref{thm-important}. 

\subsection{Yet another look at the extension class $\alpha$}\label{sec-4.1}

The sequence of filtered Lie algebroids morphism\footnote{Here the filtration put on $A$, resp.~$L$, is the constant one, 
so that the induced filtration on its universal enveloping algebra is the standard one. } 
$A\longrightarrow A^{(1)}\longrightarrow L$ provides us with a morphism of filtered $U(A)$-algebras 
$U\big(A^{(1)}\big)\longrightarrow U(L)$. It turns out to restrict to an isomorphism (of $R$-modules) 
$F^1U\big(A^{(1)}\big)\tilde\longrightarrow U(L)^{\leq1}$ between the first filtered pieces. 

Therefore we get the following isomorphism of extensions (meaning that the diagram commutes and the lines 
are exact sequences), for any $A$-module $E$: 
\begin{equation}\label{eq-diag1}
\vcenter{\xymatrix{
0\ar[r] & E\ar[r] & F^1\Big(U\big(A^{(1)}\big)\oUA E\Big)\ar[r]\ar[d] & L/A\otimes_RE\ar[r] & 0 \\
0\ar[r] & E\ar[r]\ar@{=}[u] & \Big(U(L)\oUA E\Big)^{\leq1}\ar[r] & L/A\otimes_RE\ar[r]\ar@{=}[u] & 0 }}
\end{equation}

Let us now restrict our attention to the case when $E=L/A$. We have a $1$-filtered $A$-module morphism 
$\psi:U\big(A^{(1)}\big)\oUA (L/A)\longrightarrow U^+\big(A^{(1)}\big)/U^+\big(A^{(1)}\big)A$, where 
$U^+\big(A^{(1)}\big)=\ker(\epsilon)\cong U\big(A^{(1)}\big)/R$ is equipped with the induced filtration. 
Notice that $U^+\big(A^{(1)}\big)/U^+\big(A^{(1)}\big)A\cong j^*j_!({\bf 1}_A)/F^0$. 
We also have an identification $F^1/F^0=L/A$ so that the following diagram of $A$-modules commutes 
(again, here lines are exact): 
\begin{equation}\label{eq-diag2}
\vcenter{\xymatrix{
0\ar[r] & L/A\ar[r] & F^1\Big(U\big(A^{(1)}\big)\oUA (L/A)\Big)\ar[r]\ar[d]_{\Psi} 
& (L/A)\otimes_R(L/A)\ar[r]\ar[d] & 0 \\
0\ar[r] & L/A\ar[r]\ar@{=}[u] & F^2/F^0\ar[r] & F^2/F^1\ar[r] & 0 }}
\end{equation}
We can now prove the ``if'' part of Theorem \ref{thm-important}. 
\begin{prop}\label{prop-4.3}
If the filtration of $j^*j_!({\bf 1}_A)$ splits, then $\alpha=0$. 
\end{prop}
\begin{proof}
If the filtration of $j^*j_!({\bf 1}_A)$ splits, then the bottom exact sequence in the diagram \eqref{eq-diag2} splits 
and therefore so does the top exact sequence in the same diagram. It follows from the commutativity of 
\eqref{eq-diag1} that the class of this exact sequence is $\alpha=\alpha_{L/A}$. 
\end{proof}

\subsection{A filtered morphism $j^*j_!({\bf 1}_A)\longrightarrow T_R(L/A)$}\label{sec-4.2}

We now assume that $\alpha=0$, which means that the $A$-action on $L/A$ can be lifted to an $A^{(1)}$-action. 
We therefore obtain a graded $A^{(1)}$-module structure on $T_R(L/A)$. We use the notation $\cdot$ for this action. 
For any $l\in L$ and any $P\in T_R(L/A)$ we now define $l\bullet P:=l\cdot P+\bar{l}\otimes P$, $\bar{l}$ 
being the class of $l$ in $L/A$. 
\begin{lem}\label{lem-4.4}
The operation $\bullet$ defines a filtered $FR(L)$-module structure on $T_R(L/A)$ such that: \\
\indent (i) It actually is an $A^{(1)}$-module structure. \\
\indent (ii) Its restriction to $A$ is the original $A$-module structure on $T_R(L/A)$. 
\end{lem}
\begin{proof}
First of all let us prove that $\bullet$ determines an $FR(L)$-action on $T_R(L/A)$: for $r\in R$ and $l\in L$ 
we obviously have $(rl)\bullet-=r(l\bullet-)$, and for any $P\in T_R(L/A)$, 
$$
r(l\bullet P)-l\bullet(rP)=rl\cdot P+r\bar{l}\otimes P-l\cdot(rP)-\bar{l}\otimes rP=rl\cdot P-l\cdot(rP)=l(r)P\,.
$$

We now prove that it turns out to be an $A^{(1)}$-action: for $a\in A$, $l\in L$ and $P\in T_R(L/A)$, 
\begin{eqnarray*}
[a,l]_{FR(L)}\bullet P & = & a\bullet(l\bullet P)-l\bullet(a\bullet P)
  =   [a,l]_{A^{(1)}}\cdot P+a\cdot(\bar{l}\otimes P)-\bar{l}\otimes(a\cdot P) \\
& = & [a,l]_L\cdot P+(a\cdot\bar{l})\otimes P=[a,l]_L\cdot P+\overline{[a,l]_L}\otimes P
  =   [a,l]_L\bullet P\,.
\end{eqnarray*}

Finally, the second property is obvious. 
\end{proof}
We obtain from the above lemma a filtered morphism of $A^{(1)}$-modules 
$$
\varphi:U\big(A^{(1)}\big)/U\big(A^{(1)}\big)A\longrightarrow T_R(L/A)\,,\quad P\longmapsto P\bullet 1\,.
$$
It is clear from the construction of $\varphi$ that ${\rm gr}(\varphi\circ\xi)=\pi$, which is surjective. 
Therefore ${\rm gr}(\varphi)$ is surjective, and thus $\varphi$ is surjective (because 
filtrations under consideration are exhaustive). 

\medskip

We will now prove that $\varphi$ is an isomorphism. To do so we will prove that ${\rm gr}(\varphi)$ is an isomorphism. 
We start with the following: 
\begin{lem}\label{lem-useful}
The two-sided ideal $\langle A\rangle$ generated by $A$ in $T_R(L)$ sits inside the kernel of ${\rm gr}(\xi)$. 
\end{lem}
\begin{proof}
Recall that the kernel of $\xi:U\big(FR(L)\big)\longrightarrow U\big(A^{(1)}\big)/U\big(A^{(1)}\big)A=j^*j_!({\bf 1}_A)$ 
is the sum of the two-sided ideal generated by $[a,l]_{FR(L)}-[a,l]_L$ ($a\in A$, $l\in L$) and of the left ideal generated 
by $A$. In particular, for $a\in A$, $l\in L$, $P\in F^kU\big(FR(L)\big)$ and $Q\in F^lU\big(FR(L)\big)$, we have: 
$$
P(al-la)Q\in\ker(\xi)+F^{k+l+1}U\big(FR(L)\big)\quad\textrm{and}\quad Pa\in\ker(\xi)\,.
$$
Therefore the two-sided ideal generated by $a\otimes l-l\otimes a$ ($a\in A$, $l\in L$) in $T_R(L)$ sits 
inside $\ker\big({\rm gr}(\xi)\big)$, as well as does the left ideal generated by $A$. Together they generate 
the two-sided ideal generated by $A$ inside $T_R(L)$. 
\end{proof}
It follows from the above lemma that ${\rm gr}(\xi)$ induces a surjective map 
$T_R(L)/\langle A\rangle\longrightarrow {\rm gr}\big(j^*j_!({\bf 1}_A)\big)$ such that, when composed with the 
surjective map ${\rm gr}(\varphi):{\rm gr}\big(j^*j_!({\bf 1}_A)\big)\longrightarrow T_R(L/A)$, 
it leads to the isomorphism $T_R(L)/\langle A\rangle\longrightarrow T_R(L/A)$. 
In particular, ${\rm gr}(\varphi)$ is an isomorphism. 

\section{PBW for an inclusion of Lie algebroids}\label{sec-5}

The goal of the present Section is to understand the $A$-module 
$$
i^*i_!({\bf 1}_A):=U(L)\oUA {\bf 1}_A=U(L)/U(L)A\,.
$$
According to \S~\ref{sec-alpha} $i^*i_!({\bf 1}_A)$ turns out to be a filtered $A$-module. 
We write $G^k:=(i^*i_!({\bf 1}_A)\big)^{\leq k}$, and borrow the notation from the previous Section. 

\subsection{What if the filtration of $i^*i_!({\bf 1}_A)$ splits?}\label{sec-5.1}

\begin{prop}\label{prop-if}
If the filtration of $i^*i_!({\bf 1}_A)$ splits, then $\alpha=0$. 
\end{prop}
\begin{proof}
We have a $1$-filtered $A$-module morphism $\psi:U(L)\oUA (L/A)\longrightarrow U^+(L)/U^+(L)A$, where 
$U^+(L)=\ker(\epsilon)\cong U(L)/R$ is equipped with the induced filtration. 
We then have the following commuting diagram, in which lines are exact: 
$$
\xymatrix{
0\ar[r] & L/A\ar[r] & \Big(U(L)\oUA (L/A)\Big)^{\leq1}\ar[r]\ar[d]^{\Psi} & (L/A)\otimes_R(L/A)\ar[r]\ar[d] & 0 \\
0\ar[r] & G^1/G^0\ar[r]\ar@{=}[u] & G^2/G^0\ar[r] & G^2/G^1\ar[r] & 0
}
$$
Finally, if the filtration of $i^*i_!({\bf 1}_A)$ splits then the bottom line in the above diagram splits too, 
and therefore so does the top one (of which the extension class is precisely $\alpha=\alpha_{L/A}$). 
\end{proof}

\subsection{Description of ${\rm gr}\big(i^*i_!({\bf 1}_A)\big)$}\label{sec-5.2}

We have the following coCartesian square of filtered (res.~graded) $R$-modules: 
$$
\vcenter{\xymatrix{
U\big(FR(L)\big) \ar[r]^{\xi}\ar[d] & j^*j_!({\bf 1}_A) \ar[d] \\
U(L) \ar[r] & i^*i_!({\bf 1}_A)}}\qquad\qquad
\left(\textrm{resp. }\vcenter{\xymatrix{
T_R(L) \ar[r]^{\!\!\!\!\!\!{\rm gr}(\xi)}\ar[d] & {\rm gr}\big(j^*j_!({\bf 1}_A)\big) \ar[d] \\
{\rm gr}\big(U(L)\big) \ar[r] & {\rm gr}\big(i^*i_!({\bf 1}_A)\big)
}}\right)
$$
We have already seen (Lemma \ref{lem-useful}) that the associated graded of 
$\xi:U\big(FR(L)\big)\longrightarrow j^*j_!({\bf 1}_A)$ descends to a surjective map 
$T_R(L)/\langle A\rangle\longrightarrow{\rm gr}\big(j^*j_!({\bf 1}_A)\big)$. 
In a completely similar way one can prove that associated graded of the maps 
$U\big(FR(L)\big)\longrightarrow U(L)$ and $U\big(FR(L)\big)\longrightarrow i^*i_!({\bf 1}_A)$ 
descend to surjective maps $S_R(L)\longrightarrow{\rm gr}\big(U(L))\big)$ and 
$S_R(L)/\langle A\rangle\longrightarrow{\rm gr}\big(i^*i_!({\bf 1}_A)\big)$. 

\medskip

Let us assume that the following two properties hold: \\[-0.3cm]

\indent\indent {\it The surjective $R$-linear map $S_R(L)\longrightarrow{\rm gr}\big(U(L)\big)$ is an 
isomorphism.}\hfill $(\star)$\\[-0.3cm]

\indent\indent {\it The surjective $R$-linear map 
$T_R(L)/\langle A\rangle\longrightarrow{\rm gr}\big(j^*j_!({\bf 1}_A)\big)$ is an isomorphism. }
\hfill $(\star$'$)$\\[-0.3cm]
\begin{rems}
Property $(\star)$ is known to hold if $L$ is a projective $R$-module (see \cite{R}). 
Property $(\star)$ also implies Assumption $(\O)$ that we've made from \S~\ref{sec-L/A}. 
Thanks to the results of the previous Section, Property $(\star$'$)$ follows from the assumption that $\alpha=0$.  
\end{rems}
Then coCartesianity of the diagram 
\begin{equation}\label{eq-coCart}
\vcenter{\xymatrix{
T_R(L) \ar[r]\ar[d] & T_R(L)/\langle A \rangle \ar[d] \\
S_R(L) \ar[r] & S_R(L)/\langle A \rangle
}}
\end{equation}
ensures us that the surjective $R$-linear map $S_R(L)/\langle A\rangle\longrightarrow{\rm gr}\big(i^*i_!({\bf 1}_A)\big)$ 
is an isomorphism. 

\subsection{What if $\alpha=0$?}

In this paragraph we consider the following property:\\[-0.3cm]

\indent\indent {\it The map $T_R(L/A)\longrightarrow S_R(L/A)$ 
splits as a surjective morphism of graded $A$-modules. }
\hfill $(\star\star)$\\[-0.3cm]
 
\noindent Notice that the symmetrization map $S_R(L/A)\longrightarrow T_R(L/A)$ provides a splitting whenever $\mathbb{Q}\subset R$. 
If we assume that $(\star\star)$ and $\alpha=0$ hold, then we have a sequence 
\begin{equation}\label{eq-comp}
S_R(L/A)\hookrightarrow T_R(L/A)\cong j^*j_!({\bf 1}_A)\twoheadrightarrow i^*i_!({\bf 1}_A)
\end{equation}
of filtered $A$-module morphisms. 
\begin{lem}
Passing to associated graded in \eqref{eq-comp} we get exactly the surjective map 
$S_R(L)/\langle A\rangle\longrightarrow {\rm gr}\big(i^*i_!({\bf 1}_A)\big)$ that appears in \S~\ref{sec-5.2}. 
\end{lem}
\begin{proof}
It follows from the commutativity of the diagram
$$
\xymatrix{
T_R(L/A)\ar[d]\ar[r]^{\!\!\!\!\sim} & T_R(L)/\langle A\rangle\ar[r]^{\!\!\!\!\sim}\ar[d] 
& {\rm gr}\big(j^*j_!({\bf 1}_A)\big)\ar[d] \\
S_R(L/A)\ar@<1ex>[u]\ar[r]^{\!\!\!\!\sim} & S_R(L)/\langle A\rangle\ar[r] & {\rm gr}\big(i^*i_!({\bf 1}_A)\big) 
}
$$
where the leftmost (upward) arrow is the same as the leftmost arrow in \eqref{eq-comp} (i.e.~a given splitting 
of $T_R(L/A)\twoheadrightarrow S_R(L/A)$). 
\end{proof}
If we further assume that Property $(\star)$ holds, then we have proved in \S~\ref{sec-5.2} that this map is 
an isomorphism (remember that $\alpha=0$ implies $(\star$'$)$). Therefore, we have: 
\begin{thm}\label{thm-main2}
Assume Properties $(\star)$ and $(\star\star)$ hold. 
Then $\alpha=0$ if and only if there exists an isomorphism of filtered $A$-module 
$S_R(L/A)\longrightarrow i^*i_!({\bf 1}_A)$. 
\end{thm}

\section{The case of an arbitrary $A$-module $E$}\label{sec-6}

Let now $E$ be an $A$-module, and consider the following two $A$-modules: 
$$
j^*j_!(E):=U\big(A^{(1)}\big)\oUA E \qquad\textrm{and}\qquad i^*i_!(E):=U(L)\oUA E\,.
$$
We denote by $F_E^n$ and $G_E^n$ the filtration pieces on those two filtered $A$-modules\footnote{Even though 
the filtered pieces of $U\big(A^{(1)}\big)$ and $U(L)$ are not $A$-modules, $F^n_E$ and $G^n_E$ are. Namely, 
for any $a\in A$ and any $P\otimes e$ in $F^n_E$ (resp.~$G^n_E$), 
$a(P\otimes e)=aP\otimes e=([a,P]-Pa)\otimes e=[a,P]\otimes e-P\otimes ae\in F^n_E~(\textrm{resp.}~G^n_E)$. }. 
One sees that 
$$
F^0_E=G^0_E=E\,, \quad F^1_E=G_E^1=\Big(U(L)\oUA E\Big)^{\leq1}\,,\quad\textrm{and}\quad F^1_E/F^0_E=G^1_E/G^0_E=(L/A)\otimes_RE\,.
$$
Therefore, if the filtration on either $j^*j_!(E)$ or $i^*i_!(E)$ splits then $\alpha_E=0$. 
We start with the following generalization of Theorem \ref{thm-important}: 
\begin{thm}\label{thm-E1}
Assume that $E$ is faithful\footnote{Here we mean that $E$ is faithful as an $R$-module, which ensures that 
$-\otimes_RE:A\textrm{-}mod\to A\textrm{-}mod$ is faithful (because the forgetful functor 
$A\textrm{-}mod\to R\textrm{-}mod$ is) and thus reflects exact sequences. }. 
Then both classes $\alpha$ and $\alpha_E$ vanish if and only if there 
exists an isomorphism filtered $A$-modules $\varphi_E:j^*j_!(E)\longrightarrow T_R(L/A)\otimes_RE$ such that 
${\rm gr}\big(\varphi_E\circ\xi_E)\big)=\pi_E$. 
\end{thm}
\noindent Here $\xi_E:U\big(FR(L)\big)\oR E\overset{\xi\otimes{\rm id}_E}{\longrightarrow}U\big(A^{(1)}\big)\oUA E$ 
and $\pi_E:T_R(L)\otimes_RE\overset{\pi\otimes{\rm id}_E}{\longrightarrow}T_R(L/A)\otimes_RE$. 
\begin{proof}[Sketch of Proof]
From the above we can assume that $\alpha_E=0$, which means that the $A$-module $E$ lifts to an 
$A^{(1)}$-module $E^{(1)}$. 
This allows one to construct a surjective filtered morphism of $A^{(1)}$-modules 
$$
\eta_E:j_!(E)\longrightarrow j_!({\bf 1}_A)\otimes_R E^{(1)}\,;\quad
P\otimes e\longmapsto P\cdot(1\otimes e)\,,\quad\big(P\in U\big(A^{(1)}\big)\textrm{ and }e\in E\big)\,.
$$
It is well-defined: for any $a\in A$, $(Pa)\cdot(1\otimes e)=P\cdot(a\otimes e+1\otimes ae)=P\cdot(1\otimes ae)$. 

We now prove an analog of Proposition \ref{prop-4.3} in \S~\ref{sec-4.1}. 
We have the following commutative diagram of $A$-modules in which lines are exact: 
\begin{equation}\label{eq-diag3}
\vcenter{\xymatrix{
0 \ar[r] & (L/A)\otimes_RE \ar[r]\ar@{=}[d] & F_E^2/F_E^0 \ar[r]\ar@{->>}[d] & F_E^2/F_E^1 \ar[r]\ar@{->>}[d] & 0 \\
0 \ar[r] & (L/A)\otimes_RE \ar[r] & (F^2/F^0)\otimes_RE \ar[r] & (F^2/F^1)\otimes_RE \ar[r] & 0
}}
\end{equation}
If the filtration $(F_E^n)_{n\geq 0}$ splits (in $A$-mod) then so does the top line in the above diagram, 
and thus the bottom line splits too (this is because the rightmost vertical arrow in \eqref{eq-diag3} is 
surjective). Faithfulness of $E$ ensures that $0\to F^1/F^0\to F^2/F^0\to F^2/F^1\to 0$ splits, which implies 
that $\alpha=0$ (see \S~\ref{sec-4.1}). 

Conversely, if we assume that $\alpha=0$ then by Theorem \ref{thm-important} we get a surjective morphism of filtered 
$A$-modules $\varphi_E:=(\varphi\otimes{\rm id}_E)\circ j^*\eta_E:j^*j_!(E)\longrightarrow T_R(L/A)\otimes_RE$. 
We show it is an isomorphism by proving it on the level of associated graded. One can see that 
$\langle A\rangle\otimes_R E$ lies in the kernel of ${\rm gr}(\xi_E)$. 
To conclude, one just observes that on associated graded the composed surjection
$$
\Big(T_R(L)/\langle A\rangle\Big)\otimes_RE\overset{{\rm gr}(\xi_E)}{\longrightarrow} j^*j!(E)
\overset{{\rm gr}(\varphi_E)}{\longrightarrow}T_R(L/A)\otimes_RE
$$
coincide with the standard surjection $\big(T_R(L)/\langle A\rangle\big)\otimes_RE\twoheadrightarrow T_R(L/A)\otimes_RE$, 
which is an isomorphism. 
\end{proof}

We now generalize Theorem \ref{thm-main2}: 
\begin{thm}\label{thm-E2}
Assume that $E$ is faithful and that Properties $(\star)$ and $(\star\star)$ hold. 
Then $\alpha$ and $\alpha_E$ both vanish if and only if there exists an isomorphism of 
filtered $A$-module $S_R(L/A)\otimes_RE\longrightarrow i^*i_!(E)$. 
\end{thm}
\begin{proof}[Sketch of Proof]
As before we can assume that $\alpha_E=0$, so that there exists an $A^{(1)}$-module $E^{(1)}$ such that $j^*E^{(1)}=E$. 
In particular, as we have seen in the proof of Theorem \ref{thm-E1}, there exists a surjective morphism of filtered $A$-modules 
$j^*\eta_E:j^*j_!(E)\longrightarrow j^*j_!({\bf 1}_A)\otimes_RE$. 

We first go through the analog of Proposition \ref{prop-if} in \S~\ref{sec-5.1}. 
We have the following commutative diagram of $A$-modules in which lines are exact: 
\begin{equation}\label{eq-diag4}
\vcenter{\xymatrix{
0 \ar[r] & (L/A)\otimes_RE \ar[r]\ar@{=}[d] & F_E^2/F_E^0 \ar[r]\ar[d] & F_E^2/F_E^1 \ar[r]\ar[d] & 0 \\
0 \ar[r] & (L/A)\otimes_RE \ar[r] & G_E^2/G_E^0 \ar[r] & G_E^2/G_E^1 \ar[r] & 0
}}
\end{equation}
If the filtration of $i^*i_!(E)$ splits then so does the bottom line in diagram \eqref{eq-diag4}, and thus 
the top line in the same diagram splits too. Therefore $\alpha=0$ (see the proof of Theorem \ref{thm-E1}). 

Let us now assume that $\alpha=0$. Then we have seen in the proof of Theorem \ref{thm-E1} that $j^*\eta_E$ is 
an isomorphism of filtered $A$-modules. 
\begin{lem}
${\rm gr}(j^*\eta_E)$ descends to an $A$-module isomorphism 
${\rm gr}\big(i^*i_!(E)\big)\tilde\longrightarrow{\rm gr}\big(i^*i_!({\bf 1}_A)\otimes_RE\big)$. 
\end{lem}
\begin{proof}[Proof of the Lemma]
We have two coCartesian diagrams of filtered $R$-modules 
$$
\vcenter{\xymatrix{
U\big(FR(L)\big)\oR E \ar@{->>}[r]\ar@{->>}[d] & j^*j_!(E) \ar@{->>}[d] \\
U(L)\oR E \ar@{->>}[r] & i^*i_!({\bf 1}_A)}}\qquad
\textrm{ and }\qquad\vcenter{\xymatrix{
U\big(FR(L)\big)\otimes_RE \ar@{->>}[r]\ar@{->>}[d] & j^*j_!({\bf 1}_A)\otimes_RE \ar@{->>}[d] \\
U(L)\otimes_RE \ar@{->>}[r] & i^*i_!({\bf 1}_A)\otimes_R
}}
$$
Passing to associated graded $R$-modules we get the following commutative diagram in which squares are coCartesian: 
$$
\xymatrix{
T_R(L)\otimes_RE \ar@{->>}[rd]\ar@{->>}[dd]\ar@{->>}[rr] && {\rm gr}\big(j^*j_!(E)\big)\ar[ld]^{j^*\eta_E}_{\sim}\ar@{->>}[dd] \\
& {\rm gr}\big(j^*j_!({\bf1}_A)\otimes_RE\big) \ar@{->>}[dd] & \\
S_R(L)\otimes_RE \ar@{->>}[rd]\ar@{->>}[rr]|!{[ur];[dr]}\hole && {\rm gr}\big(i^*i_!(E)\big) \\
&  {\rm gr}\big(i^*i_!({\bf 1}_A)\otimes_RE\big) & 
}
$$
Therefore we get an $R$-linear isomorphism ${\rm gr}\big(i^*i_!(E)\big)\tilde\longrightarrow {\rm gr}\big(i^*i_!({\bf 1}_A)\otimes_RE\big)$, 
which happens to be $A$-linear (this is because the square it files is made of surjective $A$-linear maps). 
\end{proof}
In particular, applying $-\otimes_RE$ to the diagram \eqref{eq-coCart} ensures that the surjective morphism of graded $R$-modules 
$S_R(L/A)\otimes_RE\cong\big(S_R(L)/\langle A\rangle\big)\otimes_RE\twoheadrightarrow {\rm gr}\big(i^*i_!(E)\big)$ is an isomorphism. 
We finally have a sequence of filtered morphisms of $A$-modules
$$
S_R(L/A)\otimes_RE\hookrightarrow T_R(L/A)\otimes_RE\overset{\varphi_E^{-1}}{\longrightarrow}
j^*j_!(E)\twoheadrightarrow i^*i_!(E)\,,
$$
for which one can check that the associated graded is the above isomorphism 
$S_R(L/A)\otimes_RE\tilde\longrightarrow{\rm gr}\big(i^*i_!(E)\big)$. 
\end{proof}

\section{Proof of the main Theorems (Theorems \ref{thm-main} and \ref{thm-E0}) and perspectives}\label{sec-7}

All what we have done so far can be sheafified. I.e.~everything remains true if we 
replace ${\bf k}$ by a sheaf of rings $\mathcal K$, $R$ by a sheaf of $\mathcal K$-algebras 
$\mathcal R$, $i:A\hookrightarrow L$ by an inclusion of sheaves of Lie algebroids 
$i:\mathcal A\hookrightarrow\mathcal L$ over $\mathcal R$, and $E$ by an $\mathcal A$-module 
$\mathcal E$.  

We now deduce Theorem \ref{thm-main} from the sheafified version of Theorem \ref{thm-main2}. 
\begin{proof}[Proof of Theorem \ref{thm-main}]
From now $\mathcal K=\underline{{\bf k}}_X$ and ${\bf k}$ is a field of zero characteristic. 
The equivalence between (1) and (2) is the sheafified version of Proposition \ref{prop-3.2}. 

Assume now that $\mathcal L$ is locally free as an $\mathcal R$-modules so that, after \cite{R}, 
Property $(\star)$ holds. Finally, ${\bf k}$ being of zero characteristic we can take the symmetrization map 
$S_{\mathcal R}(\mathcal L/\mathcal A)\longrightarrow T_{\mathcal R}(\mathcal L/\mathcal A)$ as a splitting 
of $T_{\mathcal R}(\mathcal L/\mathcal A)\longrightarrow S_{\mathcal R}(\mathcal L/\mathcal A)$ in $\mathcal A$-modules. 
Therefore Property $(\star\star)$ holds. 

We apply (the sheafified version of) Theorem \ref{thm-main2} to get the desired result. 
\end{proof}
Similarly, we deduce Theorem \ref{thm-E0} from (the sheafified version of) Theorem \ref{thm-E2} 
(one just has to notice that locally free modules are faithful). 

\subsection*{Perspectives}

Let $\imath:X\hookrightarrow Y$ be a closed embeddings of smooth algebraic varieties. 
It is shown in \cite{CCT2} (see also \cite{Yu}) that the shifted normal bundle $\mathcal A=N_{X/Y}[-1]$ is a Lie algebroid object 
in the derived category $D({\bf k}_X)$ of sheaves of ${\bf k}$-modules on $X$. 
Namely, the anchor map $N_{X/Y}[-1]\longrightarrow T_X$ is given by the normal bundle exact sequence 
$$
0\longrightarrow T_X\longrightarrow \imath^*T_Y\longrightarrow N_{X/Y}\longrightarrow 0
$$
and the Lie bracket comes from the fact that $N_{X/Y}[-1]$ is the cohomology of the relative tangent complex 
$\mathbb{T}_{X/Y}$. Moreover, it is proved in \cite{CCT2} that the universal enveloping algebra of this Lie algebroid is 
$\imath^*\imath_!\mathcal O_X$. According to \cite{AC} it satisfies the PBW condition $(\star)$ if and only if a certain class in 
$Ext_{\mathcal O_X}^2\left(\wedge^2(N_{X/Y}),N_{X/Y}\right)$ vanishes. It can be understood as the class of the following extention: 
$$
0\longrightarrow N_{X/Y}[-1]\longrightarrow U^+(\mathcal A)^{\leq2}\longrightarrow S^2\left(N_{X/Y}[-1]\right)\longrightarrow 0\,.
$$

Now let $\jmath:Y\hookrightarrow Z$ be another closed embedding of smooth algebraic varieties and consider the Lie algebroid object 
$\mathcal L=N_{X/Z}[-1]$ in $D(X)$. We have a Lie algebroid inclusion $\mathcal A\to\mathcal L$, with $\mathcal L/\mathcal A=\imath^*N_{Y/Z}[-1]$. 
It would be interesting to understand the geometric meaning of our main result in this context, and its relation to the sequence of derived self-intersections $X\times^h_YX\to X\times^h_ZX\to Y\times^h_ZY$. 

\medskip

It would also be interesting to search for a geometric interpretation of our main result in the direction of what Bordemann 
did for the Lie algebra case in \cite{Bord}. 

\section*{Appendix: sketch of proof of Proposition \ref{prop-sameclasses}}

We borrow the notation from \S~\ref{sec-alpha} and start with the following: 
\begin{lem}\label{lem:appendix}
The $A$-module $J_{L/A}^1(E)$ is isomorphic to the kernel of the difference map
\begin{equation}\label{eq:appendix}
E\oplus{\rm Hom}_R\left(L/A,\left(U(L)\oUA E\right)^{\leq1}\right)\longrightarrow {\rm Hom}_R(L/A,L/A\otimes_RE)\,.
\end{equation}
\end{lem}
\begin{proof}[Sketch of proof]
We first construct a map $\ell:J_{L/A}^1(E)\longrightarrow {\rm Hom}_R\left(L/A,\left(U(L)\oUA E\right)^{\leq1}\right)$; 
for any $\phi\in J_{L/A}^1(E)$ and any $l\in L$ we set 
$\ell(\phi)(l):=l\otimes \phi(1)-1\otimes \phi(l)\in \left(U(L)\oUA E\right)^{\leq1}$. 
First of all observe that for any $a\in A$, $\ell(\phi)(a)=a\otimes \phi(1)-1\otimes \phi(a)=a\otimes \phi(1)-1\otimes a\phi(1)=0$. 
Hence $\ell(\phi)$ factors through $L/A$. Then we show that $\ell$ is $A$-linear: for $\phi\in J_{L/A}^1(E)$, $a\in A$, and $l\in L$,  
\begin{eqnarray*}
\big(a\cdot\ell(\phi)\big)(l) 
& = & a\cdot\big(\ell(\phi)(l)\big)-\ell(\phi)([a,l]) \\
& = & al\otimes\phi(1)-a\otimes\phi(l)-[a,l]\otimes\phi(1)+1\otimes\phi([a,l]) \\
& = & la\otimes\phi(1)-a\otimes\phi(l)+1\otimes\phi([a,l]) \\
& = & l\otimes\phi(a)-1\otimes\phi(la) \\
& = & \ell(a*\phi)(l)\,.
\end{eqnarray*}
Finally, composing $\ell(\phi)$ with the epimorphism $\left(U(L)\oUA E\right)^{\leq1}\longrightarrow L/A\otimes_RE$ we get 
$l\mapsto l\otimes\phi(1)$, which coincides with the image of $\phi$ through $J_{L/A}^1(E)\to E\to {\rm Hom}_R(L/A,L/A\otimes_RE)$. 
This provides a morphism from $J_{L/A}^1(E)$ to the kernel of the difference map \eqref{eq:appendix}. 

We now construct an inverse to that map. For any $e+f$ in the kernel of \eqref{eq:appendix}, where $e\in E$ and 
$f\in{\rm Hom}_R\left(L/A,\left(U(L)\oUA E\right)^{\leq1}\right)$, we associate an element $\phi_{e,f}$ of $J_{L/A}^1(E)$ in the following way.  
For any $r\in R$ we set $\phi_{e,f}(r):=re$. If $l\in L$ then one notices that 
$f(l)-l\otimes e\in \left(U(L)\oUA E\right)^{\leq1}$ projects onto zero in $L/A\otimes_RE$ and thus has the form $1\otimes \phi_{e,f}(l)$. 
\end{proof}
\medskip
We now turn to the proof Proposition \ref{prop-sameclasses}. Let us assume that we have an extension 
$$
0\longrightarrow E\longrightarrow B\longrightarrow(L/A)\otimes_RE\longrightarrow 0
$$
representing a class $\beta\in {\rm Ext}^1_A\big((L/A)\otimes_RE,E\big)\cong{\rm Hom}_{D(A)}\big((L/A)\otimes_RE,E[1]\big)$. 
We would like to describe the image of $\beta$ through the map 
\begin{equation}\label{eq:ex-app}
{\rm Hom}_{D(A)}\big((L/A)\otimes_RE,E[1]\big)\longrightarrow
{\rm Hom}_{D(A)}\left((L/A)\overset{\mathbb{L}}{\otimes}_RE,E[1]\right)\cong
{\rm Hom}_{D(A)}\big(E,\mathbb{R}{\rm Hom}(L/A,E[1])\big)\,.
\end{equation}
One first chooses an injective resolution $\widetilde{E}$ of $E$ and considers the induced exact sequence of complexes 
$$
0\longrightarrow \widetilde E\longrightarrow \widetilde B\longrightarrow(L/A)\otimes_R\widetilde E\longrightarrow 0\,,
$$
where $\widetilde B$ is the cokernel of $E\longrightarrow \widetilde E\oplus B$; one then applies ${\rm Hom}_A(L/A,-)$ 
and get an exact sequence 
$$
0\longrightarrow {\rm Hom}_A\left(L/A,\widetilde E\right)\longrightarrow {\rm Hom}_A\left(L/A,\widetilde B\right)\longrightarrow
{\rm Hom}_A\left(L/A,(L/A)\otimes_R\widetilde E\right)\longrightarrow 0\,.
$$
One finally gets an exact sequence 
$$
0\longrightarrow {\rm Hom}_A\left(L/A,\widetilde E\right)\longrightarrow \widetilde C\longrightarrow\widetilde E\longrightarrow 0\,,
$$
where $\widetilde C$ is the kernel of 
$\widetilde E\oplus{\rm Hom}_A\left(L/A,\widetilde B\right)\longrightarrow{\rm Hom}_A\left(L/A,(L/A)\otimes_R\widetilde E\right)$. 
This defines the desired element in 
${\rm Hom}_{D(A)}\big(E,\mathbb{R}{\rm Hom}(L/A,)E[1]\big)\cong{\rm Hom}_{D(A)}\big(\widetilde E,{\rm Hom}(L/A,\widetilde E[1])\big)$. 

\begin{proof}[Proof of Proposition \ref{prop-sameclasses}]We apply the above constuction to $\beta=\alpha_E$, 
with $B=\left(U(L)\oUA E\right)^{\leq1}$. One first easily sees that 
$\widetilde B=\left(U(L)\oUA \widetilde E\right)^{\leq1}$. Then it follows from Lemma \ref{lem:appendix} that 
$\widetilde C=J^1_{L/A}(\widetilde E)$. Therefore, the image of $\alpha_E$ through the map \eqref{eq:ex-app} is determined by the 
exact sequence 
$$
0\longrightarrow {\rm Hom}_A\left(L/A,\widetilde E\right)\longrightarrow J_{L/A}^1(\widetilde E)\longrightarrow\widetilde E\longrightarrow 0\,,
$$
which is precisely the image of the class $\widetilde{\alpha}_E\in {\rm Ext}^1_A\big(E,{\rm Hom}(L/A,E)\big)$ through the map
$$
{\rm Ext}^1_A\big(E,{\rm Hom}(L/A,E)\big)\cong{\rm Hom}_{D(A)}\big(E,{\rm Hom}(L/A,E)[1]\big)\longrightarrow
{\rm Hom}_{D(A)}\big(E,\mathbb{R}{\rm Hom}(L/A,E)[1]\big)\,.
$$
The Proposition is proved. 
\end{proof}

\end{document}